\documentclass{daj}

\usepackage{amsmath,amsfonts,amsthm}
\usepackage{mathtools}

\newtheorem{theorem}{Theorem}[section]
\newtheorem{claim}[theorem]{Claim}
\newtheorem{conjecture}[theorem]{Conjecture}
\newtheorem{corollary}[theorem]{Corollary}
\newtheorem{lemma}[theorem]{Lemma}

\DeclarePairedDelimiter\set{\{}{\}}

\newcommand{\R}{\ensuremath{\mathbb{R}}}
\newcommand{\Z}{\ensuremath{\mathbb{Z}}}
\newcommand{\Ball}[1]{B_2^n(#1)}

\newcommand{\lat}{\mathcal{L}}
\DeclareMathOperator{\vol}{vol}

\dajAUTHORdetails{%
  title = {A Counterexample to a Strong Variant of the Polynomial Freiman-Ruzsa Conjecture in Euclidean Space}, %% please capitalize all significant words
  author = {Shachar Lovett and Oded Regev},
  plaintextauthor = {Shachar Lovett, Oded Regev},
  runningtitle = {A Counterexample to a Strong Variant of the PFR Conjecture}, 
}   %%% END \dajAUTHORdetails

\dajEDITORdetails{%
   year={2017},
   number={8},
   received={3 February 2017},   % received date: example: 7 January 2017
   published={9 May 2017},  % published date
   doi={10.19086/da.1640},       % XXX = number of paper, e.g. da006 for paper#6
}   %%% END \dajEDITORdetails

\begin{document}

\begin{frontmatter}[classification=text]
%% EDITOR: this will force the keywords to appear right after the Abstract.
%%   If the abstract is too long and would force the keywords off the
%%   front page, please comment out % [classification=text] above
%%   This way the keywords will be floated on the bottom of the first page
%%   even though the Abstract spills over to the next page.

%%% AUTHOR: Title goes here.  This line is optional.  You must use it
%%   if title has footnote attached or requires nontrivial typesetting,
%%   e.g., inclusion of linebreaks to force nice layout.
%\title{Short Proof of R\"odl's $n^{\log\log n}$ Bound\footnote{This is a footnote to the title}} %% please capitalize all significant words

%%% AUTHOR:
%%% List all authors. If you wish, place grant acknowledgements in \thanks.
%%% In brackets include a short tag for each author.
\author[lovett]{Shachar Lovett\thanks{Supported by an NSF CAREER award 1350481, an NSF CCF award 1614023, and a Sloan fellowship.}}
\author[regev]{Oded Regev\thanks{Supported by the Simons Collaboration on Algorithms and Geometry and by the National Science Foundation (NSF) under Grant No.~CCF-1320188. Any opinions, findings, and conclusions or recommendations expressed in this material are those of the authors and do not necessarily reflect the views of the NSF.}}

%%% AUTHOR: Abstract goes here
\begin{abstract}
The Polynomial Freiman-Ruzsa conjecture is one of the central open problems in additive combinatorics.
If true, it would give tight quantitative bounds relating combinatorial and algebraic notions of approximate subgroups.
In this note, we restrict our attention to subsets of Euclidean space. In this regime,
the original conjecture considers approximate algebraic subgroups as the set of lattice points in a convex body.
Green asked in 2007 whether this can be simplified to a generalized arithmetic progression, while
not losing more than a polynomial factor in the underlying parameters.
We give a negative answer to this question, based on a
recent reverse Minkowski theorem combined with estimates for random lattices.
\end{abstract}
\end{frontmatter}

\section{Introduction}

The Polynomial Freiman-Ruzsa conjecture, first suggested by Katalin Marton, would, if true, give a polynomial relation between combinatorial
and algebraic notions of approximate groups. In this note, we restrict our attention to subsets of Euclidean space.

Let $A$ be a finite subset of $\R^n$.
Its \emph{Minkowski sumset} is $A+A=\{a_1+a_2: a_1,a_2 \in A\}$.
The \emph{doubling factor} of $A$ is $|A+A|/|A|$.
Sets of small doubling can be viewed as a combinatorial notion of ``approximate subgroups.'' In $\R^n$, a natural example is that of lattice points in a symmetric convex body, as the following standard fact shows.

\begin{claim}
\label{clm:convex}
Let $B \subset \R^n$ be a symmetric (i.e., $B=-B$) convex body, let $\lat \subset \R^n$ be an $n$-dimensional lattice, and take $A = \lat \cap B$. Then $|A+A| \le 5^n |A|$.
\end{claim}

\begin{proof}
Note that $A+A \subseteq \lat \cap 2B$. Next, let $D$ be a maximal subset of
$\lat \cap 2B$ satisfying that the sets in $\{x+B/2 : x \in D\}$ are disjoint.
Then on one hand, by the maximality of $D$, $\lat \cap 2B \subseteq D + (\lat \cap B)$.
On the other hand, by a volume packing argument, $|D| \le 5^n$.
Therefore,
\[
|A+A| \le |\lat \cap 2B| \le |D| \cdot |A| \le 5^n |A| \; . \qedhere
\]
\end{proof}

We remark that Claim~\ref{clm:convex} fails for non-symmetric convex bodies.
Take, for instance, $B \subset \R^3$ to be the convex hull of the four points $(N,0,0),(-N,0,0),(0,N,1),(0,-N,1)$
for some large integer $N$, $\lat = \Z^3$ and set $A = \lat \cap B$. Then $|A|=4N+2$ but $|A+A| \ge (2N+1)^2$.
% It's easy to get a counterexample that's symmetric around its barycenter:
%Let B \subset R^3 be the body we currently have. Then consider the convex hull of
  %B x {0}
%with
 %-B x {1}
%which is a convex body in R^4 symmetric around its center.

Freiman~\cite{Freiman} showed that sets of small doubling must be contained in a low-dimensional affine subspace. Concretely, if $|A+A| \le K |A|$ then $A$
is supported on a subspace of dimension $2K$. However, there is an exponential gap between this bound (which is tight) and the example of lattice points
in convex bodies. The Polynomial Freiman-Ruzsa (PFR) conjecture is an attempt to bridge this gap.
One natural formulation is the following (see~\cite{GreenBlog} for a further discussion).

\begin{conjecture}[PFR conjecture in $\R^n$]
\label{conj:PFR}
There exists an absolute constant $c>0$ such that the following holds.
Let $A \subset \R^n$ be a set with $|A+A| \le K |A|$. Then for some $d \le c \log K$ there exist
\begin{enumerate}
\item a $d$-dimensional lattice $\lat \subset \R^d$,\footnote{One can restrict without loss of generality to $\lat=\Z^d$.}
\item a symmetric convex body $B \subset \R^d$,
\item a linear map $\varphi:\R^d \to \R^n$, and
\item a set $X \subset \R^n$ of size $|X| \le K^c$
\end{enumerate}
such that $|\lat \cap B| \le |A|$ and $A \subset \varphi(\lat \cap B)+X$.
\end{conjecture}

Green~\cite{GreenBlog} asked (and speculated that the answer is negative) whether Conjecture~\ref{conj:PFR} could potentially be strengthened,
where $\varphi(\lat \cap B)$ is replaced by a more restricted object, a generalized arithmetic progression (GAP).
A $d$-dimensional GAP $G \subset \R^n$ is a set of the form
\begin{equation}
\label{eq:gap}
G = \set[\Big]{x_0+\sum_{i=1}^d \alpha_i x_i:  a_i \le \alpha_i \le b_i, \; \alpha_i \in \Z} \; ,
\end{equation}
where $x_0,x_1,\ldots,x_d \in \R^n$ and $a_1,\ldots,a_d,b_1,\ldots,b_d \in \Z$. By a slight abuse of notation, the size of a GAP $G$ is its size as a multiset, namely $|G|=\prod_{i=1}^d (b_i-a_i+1)$.

As one can observe, a GAP $G$ can be written as $\varphi(\lat \cap B)+x$ as defined in Conjecture~\ref{conj:PFR}.
More precisely, assuming that $b_i-a_i$ is even for all $i$,
we can take $\lat=\Z^d$, $B=\prod_{i=1}^d [-(b_i-a_i)/2,(b_i-a_i)/2]$, $\varphi(e_i)=x_i$ where $e_1,\ldots,e_d$ are the standard unit vectors in $\R^d$, and $x=x_0+\sum_i (a_i+b_i)x_i/2$.
%As one can observe, a GAP $G$ is a special case of $\varphi(\lat \cap B)$ as defined in Conjecture~\ref{conj:PFR}.
%Concretely, $\lat=\Z^d$, $B=\prod_{i=1}^d [a_i,b_i]$ and $\varphi(e_i)=x_i$ where $e_1,\ldots,e_d$ are the standard unit vectors in $\R^d$.
Moreover, GAPs have a simpler
combinatorial structure than the general case of linear images of lattice points in a convex body. As such, it will be pleasing if Conjecture~\ref{conj:PFR} can be simplified as follows.

\begin{conjecture}[PFR conjecture in $\R^n$; GAP version]
\label{conj:PFRstrong}
There exists an absolute constant $c>0$ such that the following holds.
Let $A \subset \R^n$ be a set with $|A+A| \le K |A|$. Then there exist
\begin{enumerate}
\item a $d$-dimensional GAP $G \subset \R^n$ of dimension $d \le c \log K$;
\item a set $X \subset \R^n$ of size $|X| \le K^c$,
\end{enumerate}
such that $|G| \le |A|$ and $A \subset G+X$.
\end{conjecture}

In this note we refute Conjecture~\ref{conj:PFRstrong}. We show that the bound on $|X|$ cannot be polynomial in $K$---it is at least of the order of $K^{c' \log \log K}$
for some $c'>0$. We note that our example is of the form $\lat \cap B$ as defined in Conjecture~\ref{conj:PFR}, so it does not shed new light on the original conjecture.

Let $A,G,X$ be as in Conjecture~\ref{conj:PFRstrong}. By an averaging argument, there exists an $x \in X$ such that
\[
|A \cap (G+x)| \ge K^{-c} |A| \; .
\]
Note that $G+x$ is also a GAP. Thus, the following theorem is sufficient to rule out Conjecture~\ref{conj:PFRstrong}.
Here and below, by a ``random $n$-dimensional lattice'' we mean a lattice chosen
from the unique probability measure over the set of determinant-one lattices in $\R^n$ that is invariant under $\mathrm{SL}_n(\R)$~\cite{Siegel45}.

\begin{theorem}
\label{thm:main}
For any $c \ge 1$ the following holds.
Let $B \subset \R^n$ be a Euclidean ball of radius $n^{5/8}$
and $\lat \subset \R^n$ be a random $n$-dimensional lattice.
Set $A = \lat \cap B$, and recall from Claim~\ref{clm:convex} that its doubling factor
satisfies $K \le 5^n$.
Then with probability tending to $1$ as $n \to \infty$ over the choice of $\lat$, the following holds.
For any $d$-dimensional GAP $G \subset \R^n$ with $d \le c n$ and
$|G| \le |A|$,
\[
|A \cap G| \le n^{-n/(25 c)} |A| \le K^{-(1/100 c) \log\log K} |A|.
\]
\end{theorem}

We note that the bound on the intersection of lattice points in a convex body with a GAP in Theorem~\ref{thm:main}
is tight, up to the constant factors.
Lemma 3.33 in~\cite{tao2006additive} shows that if $A = \lat \cap B$ where $\lat$ is a lattice and $B$ is any symmetric convex body in $\R^n$,
then there exists an $n$-dimensional GAP $G \subset A$ such that $|A \cap G| \ge (c'' n)^{-7n/2} |A|$ for some constant $c''>0$.

%%%%%%%%%%%%%%%%%%%%%%%%%%%%%%

\section{Preliminaries}

A rank-$d$ \emph{lattice} $\lat \subset \R^n$ is the set of integer linear combinations of $d$ linearly independent vectors
$B = (b_1,\ldots, b_d)$,
\[
\lat = \Big\{ \sum_{i=1}^d a_i b_i \ : \ a_i \in \Z \Big\}
\; .
\]
% A lattice is \emph{full rank} if $d=n$.
The \emph{determinant} of the lattice is given by $\det(\lat) = \det(B^t B)^{1/2}$, where we view $B$ as an $n \times d$ matrix.
One can verify that the determinant is independent of the choice of basis of a lattice.
We say that a subspace $W \subset \R^n$ is
a \emph{lattice subspace} of $\lat$ if it is spanned by vectors in the lattice $\lat$. We denote by $\Ball{r}$ the Euclidean ball of radius $r$ in $\R^n$ centered at the origin.

%%%%%%%%%%%%%%%%%

\section{Properties of Random Lattices}

We will use the following lower bound on the determinants
of sublattices of a random lattice. The formulation below is due
to~\cite{ShapiraW14}, which in turn is based on the estimates
of~\cite{Thunder98}.

\begin{theorem}[{\cite[Proposition 3]{ShapiraW14}}]
\label{thm:det-lb-haar}
Let $\lat$ be a random $n$-dimensional lattice.
Then with probability tending to $1$ as $n \rightarrow \infty$, it holds
that for any lattice subspace $W$ of $\lat$,
\[
\det(\lat \cap W)^{1/\dim(W)} \geq c_{1} \cdot n^{(1-\dim(W)/n)/2} \; ,
\]
where $c_{1}>0$ denotes an absolute constant.
\end{theorem}

We will also need the following ``reverse Minkowski'' theorem,
earlier conjectured by Dadush~\cite{DR16}.

\begin{theorem}[{\cite{RegevSD16}}]
\label{thm:RM}
Let $n \ge 2$ and $\lat \subset \R^n$ be a lattice satisfying that
for any lattice subspace $W$ of $\lat$,
$\det(\lat \cap W)^{1/\dim(W)} \geq R$. Then,
for any $r \ge R$,
\[
|\lat \cap \Ball{r}| \leq (3/2) \exp(500 (\log n  \cdot r/R)^2) \; .
\]
\end{theorem}

By combining Theorems~\ref{thm:det-lb-haar} and~\ref{thm:RM}, we obtain the following.

\begin{corollary}
\label{cor:pointcountingrandom}
Let $\lat$ be a random $n$-dimensional lattice.
Then with probability tending to $1$ as $n \rightarrow \infty$, it holds
that for any $n/2$-dimensional lattice subspace $W$ of $\lat$ and any $r \ge c_{1} \cdot n^{1/4}$,
\[
|\lat \cap W \cap \Ball{r}| \leq (3/2) \exp(500 (\log n \cdot r/(c_1 n^{1/4}))^2) \; .
\]
\end{corollary}

\section{Proof of the main theorem}

We will need the following point-counting version of Minkowski's
first theorem due to Blichfeldt and van der Corput (see \cite[Thm.~1 of Ch.~2, Sec.~7]{Lekkerkerker_book}).

\begin{lemma}[\cite{vanderCorput1936}]
	\label{lem:minkowski}
For any lattice $\lat \subset \R^n$ with $\det(\lat) \leq 1$ and $r > 0$,
\begin{equation*}
|\lat \cap \Ball{r}|
\geq 2^{-n} \cdot \vol(\Ball{r}) =  \frac{1}{\sqrt{\pi n}} \Big(\frac{\pi e r^2}{2 n}  \Big)^{n/2} (1+o(1))
\; .
\end{equation*}
\end{lemma}

\begin{proof}[Proof of Theorem~\ref{thm:main}]
By Lemma~\ref{lem:minkowski} applied to $A=\lat \cap \Ball{n^{5/8}}$, if $n$ is large enough then
\begin{equation}
\label{eq:minkbound}
|A| \ge n^{n/8} \; .
\end{equation}
Let $G$ be a $d$-dimensional GAP as in Eq.~\eqref{eq:gap} where $d \le c n$.
Assume without loss of generality that the indices $i$
are sorted in non-increasing order of $b_i-a_i$.
For any $t=(t_{n/2},\ldots,t_d)$ where $a_i \le t_i \le b_i$, consider the restriction of the GAP obtained by fixing all but its first $n/2-1$ dimensions,
$$
G_t = \set[\Big]{x_0+\sum_{i=1}^{n/2-1} \alpha_i x_i + \sum_{i=n/2}^{d} t_i x_i:
a_i \le \alpha_i \le b_i, \; \alpha_i \in \Z}.
$$
Let $W_t$ be an $n/2$-dimensional linear subspace containing $G_t$.
By Corollary~\ref{cor:pointcountingrandom},
\[
|A \cap G_t| \le |\lat \cap W_t \cap \Ball{n^{5/8}}| \le (3/2) \exp(500  (\log n/c_1)^2 n^{3/4}) \le 2^n,
\]
where the last inequality assumes that $n$ is large enough. Therefore,
\begin{align*}
|A \cap G| &\le \sum_t |A \cap G_t| \\
&\le 2^n \cdot \prod_{i=n/2}^d (b_i-a_i+1)  \\
&\le 2^n \cdot |G|^{1-(n/2-1)/d}  \\
&\le 2^n \cdot |A|^{1-(n/2-1)/d} \\
&\le (2^n/|A|^{n/(3d)}) \cdot |A|.
\end{align*}
To conclude the proof note that $|A|^{n/(3d)} \ge |A|^{1/(3 c)} \ge n^{n/(24 c)}$ by Eq.~\eqref{eq:minkbound}. Therefore $|A \cap G|/|A| \le n^{-n/(25c)}$
assuming $n$ is large enough.
\end{proof}

\section*{Acknowledgment}
We thank Daniel Dadush for useful discussions. 

%%% AUTHOR:
%%% Bibliography goes here. Note that the arXiv cannot process bibtex
%%% or biber bibliographies.  Example of acceptable bibliograpy format:
%\bibliographystyle{amsplain}
%\bibliography{pfr}

\begin{thebibliography}{10}

\bibitem{DR16}
Daniel Dadush and Oded Regev, \emph{Towards strong reverse {M}inkowski-type
  inequalities for lattices}, {FOCS}, 2016,
  \url{http://arxiv.org/abs/1606.06913}.

\bibitem{Freiman}
G.~A. Fre{\u\i}man, \emph{Foundations of a structural theory of set addition},
  American Mathematical Society, Providence, R. I., 1973, Translated from the
  Russian, Translations of Mathematical Monographs, Vol 37. \MR{0360496}

\bibitem{GreenBlog}
Ben Green, 2007, A guest post on Terence Tao's blog,
  \url{https://terrytao.wordpress.com/2007/03/11/ben-green-the-polynomial-freiman-ruzsa-conjecture/}.

\bibitem{Lekkerkerker_book}
P.~M. Gruber and C.~G. Lekkerkerker, \emph{Geometry of numbers}, second ed.,
  North-Holland Mathematical Library, vol.~37, North-Holland Publishing Co.,
  Amsterdam, 1987. \MR{893813}

\bibitem{RegevSD16}
Oded Regev and Noah Stephens{-}Davidowitz, \emph{A reverse {M}inkowski
  theorem}, 2016, Available at \url{https://arxiv.org/abs/1611.05979}.

\bibitem{ShapiraW14}
Uri Shapira and Barak Weiss, \emph{A volume estimate for the set of stable
  lattices}, C. R. Math. Acad. Sci. Paris \textbf{352} (2014), no.~11,
  875--879. \MR{3268755}

\bibitem{Siegel45}
Carl~Ludwig Siegel, \emph{A mean value theorem in geometry of numbers}, Ann. of
  Math. (2) \textbf{46} (1945), 340--347. \MR{0012093}

\bibitem{tao2006additive}
Terence Tao and Van Vu, \emph{Additive combinatorics}, Cambridge Studies in
  Advanced Mathematics, vol. 105, Cambridge University Press, Cambridge, 2006.
  \MR{2289012}

\bibitem{Thunder98}
Jeffrey~Lin Thunder, \emph{Higher-dimensional analogs of {H}ermite's constant},
  Michigan Math. J. \textbf{45} (1998), no.~2, 301--314. \MR{1637658}

\bibitem{vanderCorput1936}
Johannes van~der Corput, \emph{{Verallgemeinerung einer Mordellschen
  Beweismethode in der Geometrie der Zahlen, Zweite Mitteilung}}, Acta
  Arithmetica \textbf{2} (1936), no.~1, 145--146 (ger).

\end{thebibliography}

\providecommand{\bysame}{\leavevmode\hbox to3em{\hrulefill}\thinspace}
\providecommand{\MR}{\relax\ifhmode\unskip\space\fi MR }
\providecommand{\MRhref}[2]{%
  \href{http://www.ams.org/mathscinet-getitem?mr=#1}{#2}
}
\providecommand{\href}[2]{#2}

%%% AUTHOR: Include a short description of each author following the
%%% structure below. Use the same short tags used previously.  
%%% Use \imageat{} and \imagedot{} instead of "@" and "." in
%%% email addresses-this replaces the symbols with graphics to avoid 
%%% e-mail address harvesting from the .pdf file
\begin{dajauthors}
\begin{authorinfo}[lovett]
  Shachar Lovett\\
  Computer Science Department\\
	UC San Diego\\
  slovett\imageat{}ucsd\imagedot{}edu \\
  \url{http://cseweb.ucsd.edu/~slovett/home.html}
\end{authorinfo}
\begin{authorinfo}[regev]
  Oded Regev\\
  Courant Institute of Mathematical Sciences\\
	New York University\\
%  johanh\imageat{}ktth\imagedot{}se \\
  \url{http://www.cims.nyu.edu/~regev/}
\end{authorinfo}
\end{dajauthors}

\end{document}